\documentclass[10pt]{amsart}%
\usepackage{graphicx}
\usepackage{amscd, color}
\usepackage{amsmath}
\usepackage{amsfonts}
\usepackage{amssymb}%
\providecommand{\U}[1]{\protect\rule{.1in}{.1in}}
\newtheorem{theorem}{Theorem}[section]

\newtheorem{remark}[theorem]{Remark}

\numberwithin{equation}{section}

\begin{document}

\title[On the constants for the real Bohnenblust-Hille inequality]{Some improvements on the constants for the real Bohnenblust-Hille inequality}

\author{Daniel Pellegrino \and Juan B. Seoane-Sep\'{u}lveda\textsuperscript{*}}

\address{Departamento de An\'{a}lisis Matem\'{a}tico,\newline\indent Facultad de Ciencias Matem\'{a}ticas, \newline\indent Plaza de Ciencias 3, \newline\indent Universidad Complutense de Madrid,\newline\indent Madrid, 28040, Spain.}
\email{jseoane@mat.ucm.es}

\address{Departamento de Matem\'{a}tica, \newline\indent Universidade Federal da Para\'{\i}ba, \newline\indent 58.051-900 - Jo\~{a}o Pessoa, Brazil.} \email{dmpellegrino@gmail.com}




\thanks{\textsuperscript{*}Supported by the Spanish Ministry of Science and Innovation, grant MTM2009-07848.}

\begin{abstract}

A classical inequality due to Bohnenblust and Hille states that for every $N \in \mathbb{N}$ and every
$m$-linear mapping $U:\ell_{\infty}^{N}\times\cdots\times\ell_{\infty}^{N}\rightarrow\mathbb{C}$ we have
\[
\left(  \sum\limits_{i_{1},...,i_{m}=1}^{N}\left\vert U(e_{i_{^{1}}},...,e_{i_{m}})\right\vert ^{\frac{2m}{m+1}}\right)  ^{\frac{m+1}{2m}}\leq
C_{m}\left\Vert U\right\Vert,
\]
where $C_{m}=2^{\frac{m-1}{2}}$. The result is also true for real Banach spaces. In this note we show that an adequate use of a recent new proof of
Bohnenblust-Hille inequality, due to Defant, Popa and Schwarting, combined
with the optimal constants of Khinchine's inequality (due to Haagerup) provides
quite better estimates for the constants involved in the real
Bohnenblust-Hille inequality. For instance, for $2\leq m\leq 14,$ we show that
the constants $C_{m}=2^\frac{m-1}{2}$ can be replaced by $2^{\frac{m^{2}+6m-8}{8m}}$ if $m$ is even and by $2^{\frac{m^{2}+6m-7}{8m}}$ if $m$ is odd, which substantially improve the known values of $C_{m}$. We also show that the new constants present a better asymptotic behavior.
\end{abstract}

\maketitle

\section{Preliminaries and background}

In 1931, Bohnenblust and Hille (\cite{bh}, or the more recent \cite{defant2, defant}) asserted that for every positive integer $N$ and every
$m$-linear mapping $U:\ell_{\infty}^{N}\times\cdots\times\ell_{\infty}%
^{N}\rightarrow\mathbb{C}$ we have%
\[
\left(  \sum\limits_{i_{1},...,i_{m}=1}^{N}\left\vert U(e_{i_{^{1}}%
},...,e_{i_{m}})\right\vert ^{\frac{2m}{m+1}}\right)  ^{\frac{m+1}{2m}}\leq
C_{m}\left\Vert U\right\Vert,
\]
where $C_{m}=2^{\frac{m-1}{2}}$ (actually this result also holds for real
Banach spaces). The case $m=2$ is a famous result known as Littlewood's $4/3$-inequality. It seems that the Bohnenblust-Hille inequality was overlooked
and was only re-discovered several decades later by Davie \cite{Davie} and
Kaijser \cite{Ka}. While the exponent $\frac{2m}{m+1}$ is optimal, the constant
$C_{m}=2^{\frac{m-1}{2}}$ is not. Very recently, A. Defant and P.
Sevilla-Peris \cite[Section 4]{defant2} indicated that by using Sawa's
estimate for the constant of the complex Khinchine's inequatily in Steinhaus
variables (see \cite{sawa}) it is possible to prove that $C_{m}\leq\left(
\frac{2}{\sqrt{\pi}}\right)^{m-1}$ in the complex case (this is a strong
improvement of the previous constants and it seems that these are the best
known estimates for the complex case).

The (complex and real) Bohnenblust-Hille inequality can be re-written in the
context of multiple summing multilinear operators, as we will see below.

Multiple summing multilinear mappings between Banach spaces is a recent, very important and useful nonlinear generalization of the concept of absolutely
summing linear operators. This class was introduced, independently, by Matos
\cite{collect} (under the terminology fully summing multilinear mappings) and
Bombal, P\'{e}rez-Garc\'{\i}a and Villanueva \cite{bpgv}. \qquad

Throughout this paper $X_{1},\ldots,X_{m}$ and $Y$ will stand for Banach
spaces over $\mathbb{K}=\mathbb{R}$ or $\mathbb{C}$, and $X^{\prime}$ stands for the
dual of $X$. By $\mathcal{L}(X_{1},\ldots,X_{m};Y)$ we denote the Banach space
of all continuous $m$-linear mappings from $X_{1}\times\cdots\times X_{m}$ to
$Y$ with the usual sup norm.

For $x_{1},...,x_{n}$ in $X$, let
\[
\Vert(x_{j})_{j=1}^{n}\Vert_{w,1}:=\sup\{\Vert(\varphi(x_{j}))_{j=1}^{n}%
\Vert_{1}:\varphi\in X^{\prime},\Vert\varphi\Vert\leq1\}.
\]
If $1\leq p<\infty$, an $m$-linear mapping $U\in\mathcal{L}(X_{1},\ldots
,X_{m};Y)$ is multiple $(p;1)$-summing (denoted $\Pi_{(p;1)}(X_{1}%
,\ldots,X_{m};Y)$) if there is a constant $L_{m}\geq0$ such that
\begin{equation}
\left(  \sum_{j_{1},\ldots,j_{m}=1}^{N}\left\Vert U(x_{j_{1}}^{(1)}%
,\ldots,x_{j_{m}}^{(m)})\right\Vert ^{p}\right)  ^{\frac{1}{p}}\leq L_{m}%
\prod_{k=1}^{m}\left\Vert (x_{j}^{(k)})_{j=1}^{N}\right\Vert _{w,1}
\label{lhs}%
\end{equation}
for every $N\in\mathbb{N}$ and any $x_{j_{k}}^{(k)}\in X_{k}$, $j_{k}%
=1,\ldots,N$, $k=1,\ldots,m$. The infimum of the constants satisfying
(\ref{lhs}) is denoted by $\left\Vert U\right\Vert _{\pi(p;1)}$. For $m=1$ we
have the classical concept of absolutely $(p;1)$-summing operators (see, e.g. \cite{defant3, Di}).

A simple reformulation of Bohnenblust-Hille inequality asserts that every
continuous $m$-linear form $T:X_{1}\times\cdots\times X_{m}\rightarrow
\mathbb{K}$ is multiple $(\frac{2m}{m+1};1)$-summing with $L_{m}%
=C_{m}=2^{\frac{m-1}{2}}$ (or $L_{m}=\left(  \frac{2}{\sqrt{\pi}}\right)
^{m-1}$ for the complex case, using the estimates of Defant and
Sevilla-Peris, \cite{defant2})$.$ However, in the real case the best constants known seem to
be $C_{m}=2^{\frac{m-1}{2}}$.

The main goal of this note is to obtain better constants for the
Bohnenblust-Hille inequality in the real case. For this task we will use a
recent proof of a general vector-valued version of Bohnenblust-Hille
inequality (\cite[Theorem 5.1]{defant}). The inequality of Bohnenblust-Hille is stated in \cite[Corollary 2]{defant} as a consequence of \cite[Theorem 5.1]{defant}.
The procedure of the proof of \cite[Corollary 2]{defant} allows us to obtain
quite better values than $C_{m}=2^{\frac{m-1}{2}}$. However, in this note we explore the ideas of
\cite{defant} in a different way, in order to obtain even better estimates for
the constants that can be derived from \cite[Corollary 2]{defant}. The
constants we obtain can be derived from \cite[Theorem 5.1]{defant} through an
adequate choice of variables.

Let us recall some results that we will need in this note. The first result is a
well-known inequality due to Khinchine (see \cite{Di}):

\begin{theorem}[Khinchine's inequality]\label{k}
For all $0<p<\infty$, there are constants $A_{p}$ and $B_{p}$ such that
\begin{equation}
A_{p}\left(  \sum_{n=1}^{N}\left\vert a_{n}\right\vert ^{2}\right)  ^{\frac
{1}{2}}\leq\left(  \int_{0}^{1}\left\vert \sum_{n=1}^{N}a_{n}r_{n}\left(
t\right)  \right\vert ^{p}dt\right)  ^{\frac{1}{p}}\leq B_{p}\left(
\sum_{n=1}^{N}\left\vert a_{n}\right\vert ^{2}\right)  ^{\frac{1}{2}}
\label{lpo}%
\end{equation}
for all positive integer $N$ and scalars $a_{1},...,a_{n}$ (here, $r_{n}$
denotes the $n$-Rademacher function)$.$
\end{theorem}

Above, it is clear that $B_{2}=1.$ From (\ref{lpo}) it follows that%
\begin{equation}
\left(  \int_{0}^{1}\left\vert \sum_{n=1}^{N}a_{n}r_{n}\left(  t\right)
\right\vert ^{p}dt\right)  ^{\frac{1}{p}}\leq B_{p}A_{r}^{-1}\left(  \int
_{0}^{1}\left\vert \sum_{n=1}^{N}a_{n}r_{n}\left(  t\right)  \right\vert
^{r}dt\right)  ^{\frac{1}{r}} \label{kcc}%
\end{equation}
and the product of the constants $B_{p}A_{r}^{-1}$ will appear later on
Theorem \ref{d}.

The notation of $A_{p}$ and $B_{p}$ will be kept along the paper. Now, let us recall a variation of an inequality due to Blei (see \cite[Lemma 3.1]{defant}).

\begin{theorem}[Blei, Defant et al.]\label{b}
Let $A$ and $B$ be two finite non-void index
sets, and $(a_{ij})_{(i,j)\in A\times B}$ a scalar matrix with positive
entries, and denote its columns by $\alpha_{j}=(a_{ij})_{i\in A}$ and its rows
by $\beta_{i}=(a_{ij})_{j\in B}.$ Then, for $q,s_{1},s_{2}\geq1$ with
$q>\max(s_{1},s_{2})$ we have%
\[
\left(  \sum_{(i,j)\in A\times B}a_{ij}^{w(s_{1},s_{2})}\right)  ^{\frac
{1}{w(s_{1},s_{2})}}\leq\left(  \sum_{i\in A}\left\Vert \beta_{i}\right\Vert
_{q}^{s_{1}}\right)  ^{\frac{f(s_{1},s_{2})}{s_{1}}}\left(  \sum_{j\in
B}\left\Vert \alpha_{j}\right\Vert _{q}^{s_{2}}\right)  ^{\frac{f(s_{2}%
,s_{1})}{s_{2}}},
\]
with%
\begin{align*}
w  &  :[1,q)^{2}\rightarrow\lbrack0,\infty),\text{ }w(x,y):=\frac
{q^{2}(x+y)-2qxy}{q^{2}-xy},\\
f  &  :[1,q)^{2}\rightarrow\lbrack0,\infty),\text{ }f(x,y):=\frac{q^{2}%
x-qxy}{q^{2}(x+y)-2qxy}.
\end{align*}

\end{theorem}

The following theorem is a particular case of \cite[Lemma 2.2]{defant}, for
$Y=\mathbb{K}$, using that the cotype $2$ constant of $\mathbb{K}$ is $1$,
i.e., $C_{2}(\mathbb{K})=1$ (following the notation from \cite{defant})$:$

\begin{theorem}[Defant et al]\label{d}
Let $1\leq r\leq2$, and let $(y_{i_{1},...,i_{m}%
})_{i_{1},...,i_{m}=1}^{N}$ be a matrix in $\mathbb{K}$. Then%
\begin{align*}
&  \left(  \sum\limits_{i_{1},...,i_{m}=1}^{N}\left\vert y_{i_{1}...i_{m}%
}\right\vert ^{2}\right)  ^{1/2} \leq\left(  A_{2,r}\right)  ^{m}\left(  \int\nolimits_{[0,1]^{m}}\left\vert
\sum\limits_{i_{1},...,i_{m}=1}^{N}r_{i_{1}}(t_{1})...r_{i_{m}}(t_{m}%
)y_{i_{1}...i_{m}}\right\vert ^{r}dt_{1}...dt_{m}\right)  ^{1/r},
\end{align*}
where%
\[
A_{2,r}\leq A_{r}^{-1}B_{2}=A_{r}^{-1}\text{ (because }B_{2}=1\text{).}%
\]

\end{theorem}

The meaning of $A_{2,r}$, $w$ and $f$ from the above theorems will also be
kept in the next section. Moreover, $K_{G}$ will denote the complex Grothendieck constant.

\section{Improved constants for the Bohnenblust-Hille theorem}

The main results from \cite{defant}, Theorem 5.1 and Corollary 5.2, are very
interesting vector-valued generalizations of Bohnenblust-Hille inequality. In
this note we explore the proof of \cite[Theorem 5.1]{defant} in such a way
that the constants obtained are better than those that can be derived from
\cite[Corollary 5.2]{defant}. We present here the proof of \cite[Corollary 5.2]{defant} for the particular case of the Bohnenblust-Hille inequality, with some changes (inspired
in \cite[Theorem 5.1]{defant}) which improve the constants (this task makes
the proof clearer and avoids technicalities from the arguments from
\cite[Theorem 5.1]{defant} in our particular case, with our particular choice
of the variables selected).

Following the proof of \cite[Corollary 5.2]{defant}, and using the optimal
values for the constants of Khinchine inequality (due to Haagerup) and $K_{G}$ for $C_{\mathbb{C},2}$ (see \cite{monats}), the
following estimates can be calculated for $C_{m}$:
\begin{align*}
C_{\mathbb{C},2}  &  =K_{G}\leq1,40491<\sqrt{2},\\
C_{\mathbb{C},m}  &  =2^{\frac{m-1}{2m}}\left(  \frac{C_{\mathbb{C},m-1}%
}{A_{\frac{2m-2}{m}}}\right)  ^{1-\frac{1}{m}}\text{ for }m\geq3,
\end{align*}
and%
\begin{align}
C_{\mathbb{R},2}  &  =\sqrt{2},\text{ }\\
C_{\mathbb{R},m}  &  =2^{\frac{m-1}{2m}}\left(  \frac{C_{\mathbb{R},m-1}%
}{A_{\frac{2m-2}{m}}}\right)  ^{1-\frac{1}{m}}\text{ for }m\geq3.\label{qqssw}
\end{align}
In particular, if $2\leq m\leq13$,
\begin{align}
C_{\mathbb{C},m}  &  \leq2^{\frac{m^{2}+m-6}{4m}}K_{G}^{2/m}\label{qq}\\
\text{ }C_{\mathbb{R},m}  &  \leq2^{\frac{m^{2}+m-2}{4m}}.\nonumber
\end{align}

\begin{remark}
It worths to mention that the above constants are not explicitly calculated in \cite{defant}. Since our procedure (below) will provide better constants for the real case we will not detail the above estimates.
\end{remark}

For the complex case the estimates (\ref{qqssw}) are much better than $C_{\mathbb{C},m}=2^{\frac{m-1}{2}}$ but worst than the constants $C_{\mathbb{C},m}=\left(
\frac{2}{\sqrt{\pi}}\right)  ^{m-1}$ obtained by Defant and Sevilla-Peris \cite{defant2}. However, a more appropriate use of some ideas from \cite{defant} can give better estimates for the real case, as we see in the following result.

\begin{theorem}
For every positive integer $m$ and every real Banach spaces $X_{1},...,X_{m},$
\[
\Pi_{(\frac{2m}{m+1};1)}(X_{1},...,X_{m};\mathbb{R})=\mathcal{L}%
(X_{1},...,X_{m};\mathbb{R})\text{ and }\left\Vert .\right\Vert _{\pi
(\frac{2m}{m+1};1)}\leq C_{\mathbb{R},m}\left\Vert .\right\Vert
\]
with
\begin{align*}
C_{\mathbb{R},2}  &  =2^{\frac{1}{2}}\text{ and }C_{\mathbb{R},3}=2^{\frac
{5}{6}},\\
C_{\mathbb{R},m}  &  \leq2^{\frac{1}{2}}\left(  \frac{C_{\mathbb{R},m-2}%
}{A_{\frac{2m-4}{m-1}}^{2}}\right)  ^{\frac{m-2}{m}}\text{ for }m>3.
\end{align*}
In particular, if $2\leq m\leq14$,
\begin{align*}
\text{ }C_{\mathbb{R},m}  &  \leq2^{\frac{m^{2}+6m-8}{8m}}\text{ if }m\text{
is even}\\
C_{\mathbb{R},m}  &  \leq2^{\frac{m^{2}+6m-7}{8m}}\text{ if }m\text{ is odd.}%
\end{align*}

\end{theorem}

\begin{proof}
The case $m=2$ is Littlewood's $4/3$ inequality. For $m=3$ we have $C_{3}=2^{\frac{5}{6}}$ from (\ref{qq})

The proof is done by induction, but the case $m$ is obtained as a combination
of the cases $2$ with $m-2$ instead of $1$ and $m-1$ as in \cite[Corollary
5.2]{defant}.

Suppose that the result is true for $m-2$ and let us prove for $m$. Let
$U\in\mathcal{L}(X_{1},...,X_{m};\mathbb{R})$ and $N$ be a positive integer.
For each $1\leq k\leq m$ consider $x_{1}^{(k)},...,x_{N}^{(k)}\in X_{k}$ so
that $\left\Vert (x_{j}^{(k)})_{j=1}^{N}\right\Vert _{w,1}\leq1,$ $k=1,..,m.$

Consider, in the notation of Theorem \ref{b},
\[
q=2, \;s_{1}=\frac{4}{3}, \; \text{ and }s_{2}=\frac{2(m-2)}{(m-2)+1}=\frac{2m-4}{m-1}.
\]
Thus,%
\[
w(s_{1},s_{2})=\frac{2m}{m+1}%
\]
and, from Theorem \ref{b}, we have%
\begin{align*}
&  \displaystyle \left(  \sum\limits_{i_{1},...,i_{m}=1}^{N}\left\vert U(x_{i_{1}}%
^{(1)},...,x_{i_{m}}^{(m)})\right\vert ^{\frac{2m}{m+1}}\right)  ^{(m+1)/2m} \le\\
&  \displaystyle \leq\left(  \sum\limits_{i_{1},...,i_{m-2}=1}^{N}\left\Vert \left(
U(x_{i_{1}}^{(1)},...,x_{i_{m}}^{(m)})\right)  _{i_{m-1},i_{m}=1}%
^{N}\right\Vert _{2}^{\frac{2(m-2)}{(m-2)+1}}\right)  ^{f(s_{2},\frac{4}%
{3})/\frac{2(m-2)}{(m-2)+1}}\\
&  \displaystyle \leq \left(  \sum\limits_{i_{m-1},i_{m}=1}^{N}\left\Vert \left(  U(x_{i_{1}%
}^{(1)},...,x_{i_{m}}^{(m)})\right)  _{i_{1}....,i_{m-2}=1}^{N}\right\Vert
_{2}^{\frac{4}{3}}\right)  ^{f(\frac{4}{3},s_{2})}.
\end{align*}

Now we need to estimate the two factors above. For simplicity, we write below
$dt:=dt_{1}...dt_{m-2}$.

For each $i_{m-1},i_{m}$ fixed, we have (from Theorem \ref{d}),
\begin{align*}
&  \displaystyle \left\Vert \left(  U(x_{i_{1}}^{(1)},...,x_{i_{m}}^{(m)})\right)
_{i_{1}....,i_{m-2}=1}^{N}\right\Vert _{2}^{\frac{4}{3}} \leq\\
&  \displaystyle \leq\left(  A_{2,\frac{4}{3}}^{m-2}\right)  ^{4/3}
{\textstyle\int\limits_{[0,1]^{m-2}}}
\left\vert \sum\limits_{i_{1},...,i_{m-2}=1}^{N}r_{i_{1}}(t_{1})...r_{i_{m-2}%
}(t_{m-2})U(x_{i_{1}}^{(1)},...,x_{i_{m}}^{(m)})\right\vert ^{\frac{4}{3}%
}dt\\
&  \displaystyle =\left(  A_{2,\frac{4}{3}}^{m-2}\right)  ^{4/3}
{\textstyle\int\limits_{[0,1]^{m-2}}}
\left\vert U\left(  \sum\limits_{i_{1}=1}^{N}r_{i_{1}}(t_{1})x_{i_{1}}%
^{(1)},...,\sum\limits_{i_{m-2}=1}^{N}r_{i_{m-2}}(t_{m-2})x_{i_{m-2}}%
^{(m-2)},x_{i_{m-1}}^{(m-1)},x_{i_{m}}^{(m)}\right)  \right\vert ^{\frac{4}%
{3}}dt.%
\end{align*}

Summing over all $i_{m-1,}i_{m}=1,...,N$ we obtain
\begin{flushleft}
$\displaystyle \sum\limits_{i_{m-1},i_{m}=1}^{N}\left\Vert \left(  U(x_{i_{1}}^{(1)},...,x_{i_{m}}^{(m)})\right)  _{i_{1}....,i_{m-2}=1}^{N}\right\Vert_{2}^{\frac{4}{3}} \leq $\\
$\displaystyle \left(A_{2,\frac{4}{3}}^{m-2}\right)^{4/3}{\textstyle\int\limits_{[0,1]^{m-2}}}
\sum\limits_{i_{m-1},i_{m}=1}^{N}\left\vert U\left(  \sum\limits_{i_{1}=1}^{N}r_{i_{1}}(t_{1})x_{i_{1}}^{(1)},...\sum\limits_{i_{m-2}=1}^{N}r_{i_{m-2}%
}(t_{m-2})x_{i_{m-1}}^{(m-2)},x_{i_{m}}^{(m-1)},x_{i_{m}}^{(m)}\right)\right\vert ^{\frac{4}{3}}dt.$
\end{flushleft}

Using the case $m=2$ we thus have
\begin{align*}
&  \sum\limits_{i_{m-1},i_{m}=1}^{N}\left\Vert \left(  U(x_{i_{1}}%
^{(1)},...,x_{i_{m}}^{(m)})\right)  _{i_{1}....,i_{m-2}=1}^{N}\right\Vert
_{2}^{\frac{4}{3}} \le\\
&  \leq\left(  A_{2,\frac{4}{3}}^{m-2}\right)  ^{\frac{4}{3}}%
{\textstyle\int\limits_{[0,1]^{m-2}}}
\left\Vert U\left(  \sum\limits_{i_{1}=1}^{N}r_{i_{1}}(t_{1})x_{i_{1}}%
^{(1)},...\sum\limits_{i_{m-2}=1}^{N}r_{i_{m-1}}(t_{m-2})x_{i_{m-2}}%
^{(m-2)},.,.\right)  \right\Vert _{\pi(\frac{4}{3};1)}^{\frac{4}{3}}dt\\
&  \leq\left(  A_{2,\frac{4}{3}}^{m-2}\right)  ^{\frac{4}{3}}%
{\textstyle\int\limits_{[0,1]^{m-2}}}
\left(  \left\Vert U\right\Vert \sqrt{2}\right)  ^{\frac{4}{3}}dt\\
&  =\left(  A_{2,\frac{4}{3}}^{m-2}\right)  ^{\frac{4}{3}}\left(  \left\Vert
U\right\Vert \sqrt{2}\right)  ^{\frac{4}{3}}.
\end{align*}
Hence%
\[
\left(\sum\limits_{i_{m-1},i_{m}=1}^{N}\left\Vert \left(  U(x_{i_{1}}%
^{(1)},...,x_{i_{m}}^{(m)})\right)  _{i_{1}....,i_{m-2}=1}^{N}\right\Vert
_{2}^{\frac{4}{3}}\right)  ^{\frac{3}{4}}\leq A_{2,\frac{4}{3}}^{m-2}%
\left\Vert U\right\Vert \sqrt{2}.
\]

Next we obtain the other estimate. For each $i_{1},...,i_{m-2}$ fixed, and
$dt:=dt_{m-1}dt_{m},$ we have (from Theorem \ref{d}):
\begin{align*}
&  \left\Vert \left(  U(x_{i_{1}}^{(1)},...,x_{i_{m}}^{(m)})\right)
_{i_{m-1},i_{m}=1}^{N}\right\Vert _{2} \le\\
&  \leq A_{2,s_{2}}^{2}\left(
{\textstyle\int\limits_{[0,1]^{2}}}
\left\vert \sum\limits_{i_{m-1},i_{m}=1}^{N}r_{i_{m-1}}(t_{m-1})r_{i_{m}%
}(t_{m})U(x_{i_{1}}^{(1)},...,x_{i_{m}}^{(m)})\right\vert ^{s_{2}}dt\right)
^{1/s_{2}}\\
&  =A_{2,s_{2}}^{2}\left(
{\textstyle\int\limits_{[0,1]^{2}}}
\left\vert U\left(  x_{i_{1}}^{(1)},...,x_{i_{m-2}}^{(m-2)},\sum
\limits_{i_{m-1}=1}^{N}r_{i_{m-1}}(t_{m-1})x_{i_{m-1}}^{(m-1)},\sum
\limits_{i_{m}=1}^{N}r_{i_{m}}(t_{m})x_{i_{m}}^{(m)}\right)  \right\vert
^{s_{2}}dt\right)  ^{1/s_{2}}.%
\end{align*}
Summing over all $i_{1},....,i_{m-2}=1,...,N$ we get:
\begin{align*}
&  \displaystyle \sum\limits_{i_{1},...,i_{m-2}=1}^{N}\left\Vert \left(  U(x_{i_{1}}%
^{(1)},...,x_{i_{m}}^{(m)})\right)  _{i_{m}=1}^{N}\right\Vert _{2}^{s_{2}}\le\\
&  \leq A_{2,s_{2}}^{2s_{2}}
{\textstyle\int\limits_{[0,1]^{2}}}
\displaystyle \sum\limits_{i_{1},...,i_{m-2}=1}^{N}\left\vert U\left(  x_{i_{1}}%
^{(1)},...,x_{i_{m-2}}^{(m-2)},\sum\limits_{i_{m-1}=1}^{N}r_{i_{m-1}}%
(t_{m-1})x_{i_{m-1}}^{(m-1)},\sum\limits_{i_{m}=1}^{N}r_{i_{m}}(t_{m}%
)x_{i_{m}}^{(m)}\right)  \right\vert ^{s_{2}}dt.
\end{align*}
We thus have, by the induction step,
\begin{align*}
&  \displaystyle \sum\limits_{i_{1},...,i_{m-2}=1}^{N}\left\Vert \left(  U(x_{i_{1}}%
^{(1)},...,x_{i_{m}}^{(m)})\right)  _{i_{m}=1}^{N}\right\Vert _{2}^{s_{2}} \le\\
&  \displaystyle \leq\left(  A_{2,s_{2}}^{2}\right)^{s_{2}} \le%
{\textstyle\int\limits_{[0,1]}}
\displaystyle \left\Vert U\left(  .,...,\sum\limits_{i_{m}-1=1}^{N}r_{i_{m-1}}%
(t_{m-1})x_{i_{m-1}}^{(m-1)},\sum\limits_{i_{m}=1}^{N}r_{i_{m}}(t_{m}%
)x_{i_{m}}^{(m)}\right)  \right\Vert _{\pi(s_{2};1)}^{s_{2}}dt\\
&  \leq\left(  A_{2,s_{2}}^{2}\right)^{s_{2}}%
{\textstyle\int\limits_{[0,1]}}
C_{\mathbb{R},m-2}^{s_{2}}\left\Vert U\right\Vert ^{s_{2}}dt =\left(  A_{2,s_{2}}^{2}\right)  ^{s_{2}}C_{\mathbb{R},m-2}^{s_{2}%
}\left\Vert U\right\Vert ^{s_{2}}%
\end{align*}

and so%
\begin{align*}
&  \left(  \sum\limits_{i_{1},...,i_{m-2}=1}^{N}\left\Vert \left(  U(x_{i_{1}%
}^{(1)},...,x_{i_{m}}^{(m)})\right)  _{i_{m}=1}^{N}\right\Vert _{2}^{s_{2}%
}\right)^{1/s_{2}} \leq\left(  A_{2,s_{2}}^{2}\right)  C_{\mathbb{R},m-2}\left\Vert
U\right\Vert.
\end{align*}
Hence, combining both estimates, we obtain
$$ \left(  \sum\limits_{i_{1},...,i_{m}=1}^{N}\left\vert U(x_{i_{1}}%
^{(1)},...,x_{i_{m}}^{(m)})\right\vert ^{\frac{2m}{m+1}}\right)  ^{(m+1)/2m}  \leq\left[  A_{2,\frac{4}{3}}^{m-2}\sqrt{2}\left\Vert U\right\Vert \right]
^{f(\frac{4}{3},s_{2})}\left[  \left(  A_{2,s_{2}}^{2}\right)  C_{\mathbb{R}%
,m-2}\left\Vert U\right\Vert \right]  ^{f(s_{2},\frac{4}{3})}.$$
Although,
\begin{align*}
f(\frac{4}{3},s_{2})  &  =\frac{4\frac{4}{3}-2\frac{4}{3}\left(  \frac
{2m-4}{m-1}\right)  }{4\left(  \frac{4}{3}+\frac{2m-4}{m-1}\right)  -4\frac
{4}{3}\left(  \frac{2m-4}{m-1}\right)  }=\frac{2}{m},\\
f(s_{2},\frac{4}{3})  &  =\frac{4\left(  \frac{2m-2}{m}\right)  -2\left(
\frac{2m-2}{m}\right)  }{4\left(  1+\frac{2m-2}{m}\right)  -4\left(
\frac{2m-2}{m}\right)  }=1-\frac{2}{m},%
\end{align*}
and, therefore%
\begin{align*}
\left(  \sum\limits_{i_{1},...,i_{m}=1}^{N}\left\vert U(x_{i_{1}}%
^{(1)},...,x_{i_{m}}^{(m)})\right\vert ^{\frac{2m}{m+1}}\right)  ^{(m+1)/2m} &  \leq\left[  A_{2,\frac{4}{3}}^{m-2}\sqrt{2}\left\Vert U\right\Vert \right]
^{\frac{2}{m}}\left[  \left(  A_{2,\frac{2m-4}{m-1}}^{2}\right)
C_{\mathbb{R},m-2}\left\Vert U\right\Vert \right]  ^{1-\frac{2}{m}}\\
&  =\left[  A_{2,\frac{4}{3}}^{m-2}\sqrt{2}\right]  ^{\frac{2}{m}}\left(
A_{2,\frac{2m-4}{m-1}}^{2}\right)  ^{1-\frac{2}{m}}C_{\mathbb{R},m-2}%
^{1-\frac{2}{m}}\left\Vert U\right\Vert \\
&  =\left[  A_{2,\frac{4}{3}}^{m-2}\sqrt{2}\right]  ^{\frac{2}{m}}\left(
A_{2,\frac{2m-4}{m-1}}^{2}C_{\mathbb{R},m-2}\right)  ^{1-\frac{2}{m}%
}\left\Vert U\right\Vert \\
&  =2^{\frac{1}{m}}\left(  A_{2,\frac{4}{3}}A_{2,\frac{2m-4}{m-1}}\right)
^{\frac{2m-4}{m}}\left(  C_{\mathbb{R},m-2}\right)  ^{\frac{m-2}{m}}\left\Vert
U\right\Vert .
\end{align*}

Now let us estimate the constants $C_{\mathbb{R},m}.$ We know that $B_{2}=1$
and, from \cite{haag}, we also know that $A_{p}=2^{\frac{1}{2}-\frac{1}{p}}%
$whenever $p\leq 1.847.$ So, for $2\leq m\leq14$ we have%
\[
A_{\frac{2m-4}{m-1}}=2^{\frac{1}{2}-\frac{m-1}{2m-4}}.
\]

Hence, from (\ref{kcc}) and using the best constants of Khinchine's inequality
from \cite{haag}, we have%

\[
\begin{array}
[c]{c}%
A_{2,\frac{4}{3}}\leq A_{\frac{4}{3}}^{-1}=2^{\frac{3}{4}-\frac{1}{2}},\\
A_{2,\frac{2m-4}{m-1}}\leq A_{\frac{2m-4}{m-1}}^{-1}=2^{\frac{m-1}{2m-4}%
-\frac{1}{2}},
\end{array}
\]

and%
$$
C_{\mathbb{R},m} \leq 2^{\frac{1}{m}}\left(  \left(  2^{\frac{3}{4}%
-\frac{1}{2}}\right)  \left(  2^{\frac{m-1}{2m-4}-\frac{1}{2}}\right)
\right)  ^{\frac{2m-4}{m}}\left(  C_{\mathbb{R},m-2}\right)  ^{1-\frac{2}{m}%
} = 2^{\frac{m+2}{2m}}\left(  C_{\mathbb{R},m-2}\right)  ^{1-\frac{2}{m}}%
$$
and we obtain, if $2\leq m\leq14$,
\begin{align*}
\text{ }C_{\mathbb{R},m}  &  \leq2^{\frac{m^{2}+6m-8}{8m}}\text{ if }m\text{
is even,}\\
C_{\mathbb{R},m}  &  \leq2^{\frac{m^{2}+6m-7}{8m}}\text{ if }m\text{ is odd.}%
\end{align*}

In general we easily get $$C_{\mathbb{R},m}\leq2^{\frac{1}{2}}\left(  \frac{C_{\mathbb{R},m-2}%
}{A_{\frac{2m-4}{m-1}}^{2}}\right)  ^{\frac{m-2}{m}}.$$

The numerical values of $C_{\mathbb{R},m}$, for $m>14$, can be easily calculated by using the exact values of $A_{\frac{2m-4}{m-1}}$, when $m>14$ (see \cite{haag}): $$A_{\frac{2m-4}{m-1}}=\sqrt{2}\left(  \frac{\Gamma\left(  \frac{\frac{2m-4}{m-1}+1}{2}\right)  }{\sqrt{\pi}%
}\right)  ^{(m-1)/(2m-4)}.$$

\end{proof}

In the below table we compare the first constants $C_{m}=2^{\frac{m-1}{2}}$ and the constants that can be derived from \cite[Cor. 5.2]{defant} with the {\em new} constants $C_{\mathbb{R},m}$:

\begin{center}
\begin{tabular}{c|c|c|c}
& {\sf New Constants} & Constants from \cite[Cor. 5.2]{defant}) &  $C_{m}=2^{\frac{m-1}{2}}$\\
\hline
$m=3$  & $2^{20/24} \approx 1.782$ & $2^{5/6}\approx 1.782$ & $2^{2/2}=2$\\
$m=4$  & $2^{32/32}=2$ & $2^{\frac{18}{16}}\approx 2.18$ &  $2^{3/2} \approx 2.828$\\
$m=5$  & $2^{\frac{48}{40}}\approx 2.298$ & $2^{\frac{36}{20}}\approx 2.639$ & $2^{2}=4$\\
$m=6$  & $2^{\frac{64}{48}}\approx 2.520$ & $2^{\frac{40}{24}}\approx 3.17$ & $2^{5/2} \approx 5.656$\\
$m=7$  & $2^{\frac{84}{56}}\approx 2.828$ & $2^{\frac{54}{28}}\approx 3.807$ & $2^{6/2}=8$\\
$m=8$  & $2^{\frac{104}{64}}\approx 3.084$ & $2^{\frac{70}{32}}\approx 4.555$ & $2^{7/2} \approx 11.313$\\
$m=9$  & $2^{\frac{128}{72}}\approx3.429$ & $2^{\frac{88}{36}}\approx5.443$ & $2^{8/2}=16$\\
$m=10$ & $2^{\frac{152}{80}}\approx3.732$ & $2^{\frac{108}{40}}\approx6.498$ & $2^{9/2} \approx 22.627$\\
$m=11$ & $2^{\frac{180}{88}}\approx4.128$ & $2^{\frac{130}{44}}\approx7.752$ & $2^{10/2}=32$\\
$m=12$ & $2^{\frac{208}{96}}\approx4.490$ & $2^{\frac{154}{48}}\approx9.243$ & $2^{11/2} \approx 45.254$\\
$m=13$ & $2^{\frac{240}{104}}\approx4.951$ &$2^{\frac{180}{52}}\approx11.016$ & $2^{12/2}=64$\\
$m=14$ & $2^{\frac{272}{112}}\approx5.383$ & $2^{\frac{13}{28}}\left(\frac{2^{\frac{180}{52}}}{A^{2}_{\frac{26}{14}}}\right)^{1-\frac{1}{14}}\approx 13.457$ &
$2^{13/2} \approx 90.509$
\end{tabular}
\end{center}

In the last line of the above table we have used that $A_{\frac{26}{14}}=\sqrt{2}\left(  \frac{\Gamma\left(  \frac{\frac{26}{14}+1}{2}\right)  }{\sqrt{\pi}%
}\right)  ^{14/26}\approx 0.9736$.

\bigskip
\section{Comparing the asymptotic behavior of the different constants}
In this final section we show that the new constants obtained present a better asymptotic behavior than the previous (including the constants derived from \cite[Cor. 5.2]{defant}).

We have seen that
\[
C_{\mathbb{R},m}\leq2^{\frac{1}{2}}\left(  \frac{C_{\mathbb{R},m-2}%
}{A_{\frac{2m-4}{m-1}}^{2}}\right)  ^{\frac{m-2}{m}}.
\]
As $m\rightarrow\infty$ we know that $A_{2,\frac{2m-4}{m-1}}$ increases to $1$. So,
\[
\lim\sup\frac{C_{\mathbb{R},m}}{\left(  C_{\mathbb{R},m-2}\right)
^{\frac{m-2}{m}}}\leq2^{\frac{1}{2}}.
\]\label{jyu}
For the original constants $C_{m}=2^{\frac{m-1}{2}}$ we have%
\[
\frac{C_{m}}{\left(  C_{m-2}\right)  ^{\frac{m-2}{m}}}=2^{\frac{2m-3}{m}}%
\]
and thus
\[
\lim\frac{C_{m}}{\left(  C_{m-2}\right)  ^{\frac{m-2}{m}}}=4.
\]
and for the constants from \cite[Cor. 5.2]{defant} a similar calculation shows us that
$2^{\frac{1}{2}}$ is replaced by $2$ in (\ref{jyu}).

In resume, the new constants, besides smaller than the others, are those which have the best asymptotic behavior.
\bigskip

{\bf Acknowledgements.} The authors thank A. Defant and P. Sevilla-Peris for important remarks
concerning the complex case.

\end{document}